\documentclass{amsart}

\usepackage[colorlinks=true,linkcolor=blue,urlcolor=blue, citecolor=blue]%
{hyperref}
\usepackage{tikz-cd}
\usepackage{graphicx}
\usepackage{pb-diagram}
\usepackage{colordvi}
\usepackage{graphics,eepic}
\usepackage[all]{xy}
\usepackage{mathrsfs}
\usepackage[centertags]{amsmath}
\usepackage{amsfonts}
\usepackage{amssymb}
\usepackage{amsthm}

\newtheorem{thm}{Theorem}
\newtheorem*{thm*}{Theorem}
\newtheorem{cor}{Corollary}
\newtheorem{lem}[cor]{Lemma}

\theoremstyle{definition}

\theoremstyle{remark}
\newtheorem{rem}[cor]{Remark}

\makeatletter
\newtheoremstyle{indented}
  {}
  {}
  {\addtolength{\@totalleftmargin}{1.5em}
   \addtolength{\linewidth}{-3.5em}
   \parshape 1 2.5em \linewidth}
  {}
  {\bfseries}
  {.}
  {.5em}
  {}
\makeatother
\theoremstyle{indented}
\newtheorem{example}{Example}[section]

\numberwithin{equation}{section}

\newcommand{\Gcal}{{\mathcal{G}}}

\newcommand{\N}{\mathbb{N}}

\newcommand{\C}{\mathbb{C}}
\newcommand{\R}{\mathbb{R}}

\newcommand{\bi}{\begin{itemize}}
\newcommand{\ei}{\end{itemize}}
\newcommand{\bd}{\begin{description}}
\newcommand{\ed}{\end{description}}
\newcommand{\beq}{\begin{equation}}
\newcommand{\eeq}{\end{equation}}
\newcommand{\beqn}{\begin{eqnarray}}
\newcommand{\eeqn}{\end{eqnarray}}
\newcommand{\beqna}{\begin{eqnarray*}}
\newcommand{\eeqna}{\end{eqnarray*}}

\newcommand{\V}{\mathcal{V}}

\newcommand{\Z}{\mathbb Z}

\def\vol{{\rm vol} }

\def\Id{{\rm Id}}
\def\NJ{{\rm NJ}}

\newcommand{\leftexp}[2]{\,{\vphantom{#2}}^{#1}\!{#2}}
\newcommand{\trnsp}[1]{\leftexp {t}{#1}}
\DeclareMathOperator{\GL}{GL}
\DeclareMathOperator{\OO}{O}
\DeclareMathOperator{\UU}{U}

\DeclareMathOperator{\tr}{tr}
\DeclareMathOperator{\sym}{sym}
\DeclareMathOperator{\diag}{diag}
\newcommand{\I}{{\mathcal {J}}}

\newcommand{\nn}{\nonumber}
\newcommand{\cP}{{\mathcal {P}}}
\newcommand{\F}{{\mathcal {F}}}

\newcommand{\gln}{\GL_n(\R)}
\newcommand{\glnc}{\GL_n(\C)}
\newcommand{\sln}{SL_n(\R)}
\def\Vb{{\mathbb V}}
\def\On{{\OO_n(\R)}}
\def\Ok{{\OO_k(\R)}}
\def\Onk{{\OO_{n-k}(\R)}}
\newcommand{\Ort}[1]{\OO(#1)}
\newcommand{\grass}[2]{\mathbb{G}_{#1,#2}}
\def\G2{\grass{n}{2}}

\def\Hom{\mathrm{Hom}}

\begin{document}

\title[Random and mean exponents for $\gln$]{Random and mean  Lyapunov
exponents for $\gln$}

\author[Armentano]{Diego Armentano }
\address{Centro de Matemática, Facultad de Ciencias, Universidad de la Rep\'ublica, Igu\'a 4225, 11400 Montevideo, Uruguay}
\email{diego@cmat.edu.uy}

\author[Chinta]{Gautam Chinta}
\address{Department of Mathematics, The City College of New York, New York, NY 10031,}
\email{gchinta@ccny.cuny.edu}

\author[Sahi]{Siddhartha Sahi}
\address{Department of Mathematics, Rutgers University, Hill Center - Busch
Campus, 110 Frelinghuysen Road Piscataway, NJ 08854-8019, USA}
\email{sahi@math.rugers.edu}
\thanks{The research of S. Sahi was partially supported by NSF grants
DMS-1939600 and 2001537, and Simons Foundation grant 509766.}

\author[Shub]{Michael Shub}
\address{Department of Mathematics, The City College and the Graduate
Center of CUNY New York.}
\email{shub.michael@gmail.com}
\thanks{Shub's research was partially supported by a grant from the
Smale Institute.}



\begin{abstract}
We consider orthogonally invariant probability measures on $\gln$ and
  compare the mean of the logs of the moduli of eigenvalues of the
  matrices to the Lyapunov exponents of random matrix products
  independently drawn with respect to the measure. We give a lower bound
  for the former in terms of the latter. The results are motivated by
  Dedieu-Shub\cite{DS}. A novel feature of our treatment is the use of the theory of spherical polynomials in the proof of our main result.
\end{abstract}

\maketitle

\section{Introduction and main result}

In this paper we investigate bounds for the \emph{mean Lyapunov exponents} for a measure on $\gln$ in terms of \emph{random Lyapunov exponents}.  To explain this further, fix a probability measure $\mu$ on $G=\gln$ or $\glnc$.  If $\mu$ satisfies a mild integrability condition, Oseledets theorem guarantees the existence of $n$ 
real numbers 
$$r_1\geq r_2\geq \cdots \geq r_n
$$
such that for almost every sequence $g_1, g_2, \ldots$ of i.i.d. matrices drawn from the measure $\mu$, the limit
\begin{equation}
  \label{eq:limit}
\lim\frac 1m \log\|g_m\cdots g_1v\|  
\end{equation}
exists for every nonzero vector $v$ and is equal to one of the $r_i$.  We call these $r_i$ the random Lyapunov exponents associated to the measure $\mu$.
If the measure $\mu$ is concentrated on a single matrix $A\in G$ the $r_i$ are simply 
$$\log|\lambda_1(A)|\geq \log|\lambda_2(A)|\geq \cdots \geq \log|\lambda_n(A)|
$$
for $\lambda_i(A)$ the eigenvalues of $A$ written according to their algebraic multiplicity.

In the complex case the main Theorem of Dedieu and Shub \cite{DS} is:
\begin{thm*}[Theorem 1, \cite{DS}]
 If $\mu$ is a unitarily invariant measure on $\glnc$, satisfying  the
  integrability condition 
  $$
  A\in\glnc\mapsto \log^{+}(\|A\|)\mbox{ and }
  \log^{+}(\|A^{-1}\|)\quad\mbox{are}\quad\mu\mbox{-integrable},
  $$
  then
  $$
  \int_{A\in\glnc} \sum_{i=1}^k\log|\lambda_i(A)|\,d\mu(A)\geq
  \sum_{i=1}^k r_i.
  $$
\end{thm*}
We note that we use the same symbol $\|\cdot\|$ for both the operator norm $\|A\|$ of a matrix and for the euclidean norm of a vector, as in (\ref{eq:limit}).  We hope no confusion will arise.
In the Theorem above we have also introduced the notation  $f^+(x)=\max\{f(x), 0\}$ for a real-valued function $f$.

In \cite{DS} and \cite{BPSW} it is asked if a similar theorem holds for
$\gln$ and $\On$ perhaps with a constant $c_n$ depending on $n$. Here we
prove that it does.
Our main theorem is the following.
\begin{thm}\label{thm:main}
There exists a universal constant $c_n>0$ such that, if 
  $\mu$ is an orthogonally invariant measure on $\gln$
  satisfying  the integrability condition 
  $A\in\gln \mapsto \log^{+}(\|A\|)\;\mbox{and}\; \log^{+}(\|A^{-1}\|)\in
  L^1(\gln,\mu)$,
  then 
  $$
  \int_{A\in\gln} \left(\sum_{i=1}^k\log|\lambda_i(A)|\right)^{+}\,d\mu(A)\geq
  c_n\left(\sum_{i=1}^k r_i\right)^+
$$
for any $k, 1\leq k\leq n.$
\end{thm}
Let $\sln$ be the \emph{special linear group} of $n\times n$ matrices
with determinant $1$. Then we have the following result.
\begin{cor}
There exists a universal constant $c'_n>0$ such
  that, if  $\mu$ is an orthogonally invariant probability measure on $\sln$, 
  then,   $$
  \int_{A\in\sln} \sum_{i=1}^k \log |\lambda_i(A)| \,d\mu(A)\\
  \geq c'_n\left(\sum_{i=1}^k r_i\right)^+. 
  $$
  \qed
\end{cor}
The proof of the corollary follows immediately since, for all $A\in
SL(n,\R)$, $\prod_{j=1}^k|\lambda_j(A)|\geq 1$, $(k=1,\ldots,n)$.

Some special cases of our main result Theorem \ref{thm:main} have been previously established. For $n=2$ the result is proved in \cite{DS} and Avila-Bochi \cite{AB}.  Rivin \cite{R} proves the case $n>2, k=1$.  (Both \cite{AB} and \cite{R} prove more general results in these restricted settings, from which the stated  results can be derived.)

\subsection{Motivation}
We place our results in a more general setting in order to provide motivation, 
which originates with the study of the entropy of
diffeomorphisms of closed manifolds.  Let $\pi:\V\to X$ be a finite-dimensional vector bundle.  The basic object of interest is the iteration of fiberwise linear maps $\mathcal A$  of $\pi$ which cover a map $f:X\to X$ of the base.
The \emph{cocycle} is described in the following diagram by the bundle map $\mathcal A:\V\to\V$ which satisfies  $\pi\circ\mathcal A=f\circ \pi$.
\begin{equation}\label{def:cocycle}
\begin{tikzcd}
\V \arrow[r,"\mathcal A"] \arrow[d,swap,"\pi"] &
  \V \arrow[d,"\pi"] \\
X\arrow[r,"f"] & X
\end{tikzcd}
\end{equation}
See Ruelle \cite{Ru}, Ma\~ne \cite{M}, and Viana \cite{V} for extensions.
We give four basic examples of this setup.

\begin{example}\label{ex:1}
  The base $X$ is one point. (This is the object of our paper.)
\end{example}

\begin{example}\label{ex:derivative}
  $X$ is a closed manifold $M$, $\V$ is the tangent bundle $TM$ of
  $M$, $f$ is a smooth (at least $C^{1+\alpha}$ endomorphism of $M$) 
and $\mathcal A=Tf$, the derivative of $f$.  This is the \emph{derivative cocycle}.
 Note that the $k$th-iterate of $Tf$ is given by
    $$
    (Tf)^k(x,v)=(f^k(x),Tf(f^{k-1}(x))\cdots Tf(x)v),\qquad (x,v)\in
    TM.
    $$\end{example}

\begin{example}\label{ex:random}
  Let $\V\xrightarrow[]{\pi}X$ be a fixed vector bundle  and $\mathscr
    F$ a family of bundle maps $(\mathcal A,f)$ as 
    in (\ref{def:cocycle}), with $\mathcal A:\V\to\V$ fibrewise linear and $f:X\to X$ a base map.  
Assume given a finite measure $\mu$ on $\mathscr F$.

    Then \emph{random products} of independent elements of $\mathscr F$,
    drawn with respect to the measure $\mu$, are
    described by the following cocycle.
    Let $\Gcal =\mathscr F^\N$ with the product measure $\mu^\N$.
    Writing elements of $\Gcal$ as 
  $$(\mathcal{A}_i,f_i)_{i}=(\cdots,(\mathcal{A}_{n}, f_n) ,\cdots,(\mathcal{A}_0,f_0))$$ 
we  define 
    $\sigma:\Gcal\to \Gcal$ by
    $\sigma ((\mathcal{A}_i,f_i)_{i})=(\mathcal{A}_{i+1},f_{i+1})_{i}$,
that is, shift to the right and delete the first term.
Then, the map $\mathcal H:\Gcal\times \V \to \Gcal\times \V$, given
    by 
    $$
    \mathcal H((\mathcal{A}_i,f_i)_i,v)=
  (\sigma((\mathcal A_i,f_i)_i),\mathcal{A}_0(v)),
\qquad ((\mathcal A_i,f_i)_i,v)\in\Gcal\times \V
    $$  
    defines the cocycle
$$
\begin{tikzcd}
  \Gcal\times \V \arrow[r,"\mathcal H"] \arrow[d,swap,"\Id_{\Gcal}\times\pi"] &
  \Gcal\times \V \arrow[d,"\Id_{\Gcal}\times\pi"] \\
  \Gcal\times X \arrow[r,"h"] & \Gcal\times X
\end{tikzcd}
$$
where the base map $h:\Gcal\times X\to \Gcal\times X$ is given by
    $h((\mathcal{A}_i,f_i)_i,x)=(\sigma((\mathcal A_i,f_i)_i),f_0(x))$,
    where $\pi(v)=x$.

The $k^{th}$-iterate of the cocycle $\mathcal H$, is given by
$$
    \mathcal{H}^k((\mathcal{A}_i,f_i)_i,v)=
    (\sigma^k((\mathcal{A}_i,f_i)_i),\mathcal{A}_{k-1}\cdots\mathcal{A}_{0}(v)),\qquad 
   ((\mathcal A_i,f_i)_i,v) \in \Gcal\times\V,
$$
   which yields the products of random  i.i.d. elements of the measure space
    $(\mathscr F,\mu)$.  
\end{example}

\begin{example}\label{ex:linear}
  Let $f:X\to X$ and $\phi:X\to\gln$.
  Let 
\begin{equation}\label{dec:linco}
\begin{tikzcd}
 X\times \R^n \arrow[r,"\mathcal A "] \arrow[d,swap,"\pi"] &
  X\times \R^n \arrow[d,"\pi"] \\
 X \arrow[r,"f"] &  X
\end{tikzcd}
\end{equation}
be defined by $\mathcal{A}(x,v)=(f(x),\phi(x)v)$.
    The functions $f$ and $\phi$ are frequently called \emph{linear
    cocyles} in the literature and $\mathcal A$ the associated linear
    extension.
    Here we use linear cocycle (or just cocycle) for both.
    In this case the $k$-th iterate of $\mathcal A$ is given by
    $$
    \mathcal {A}^k(x,v)=(f^{k}(x),\phi(f^{k-1}(x)) \cdots\phi(f(x))
    \phi(x)v),\qquad (x,v)\in X\times \R^n.
    $$
\end{example}

\medskip
\medskip
We now return to the general setting of a finite dimensional vector
bundle $\V\xrightarrow[]{\pi}X$  and cocycle as in
(\ref{def:cocycle}).  Assume that $\pi$ has a \emph{Finsler structure},
i.e. a norm on each fiber of $\V.$
Consider the limit
\begin{equation}\label{def:Lyexp}
  \lim_n \frac1n \log \frac{\|\mathcal{A}^n(v)\|}{\|v\|}, 
\end{equation}
for a  given nonzero vector $v\in \V$.
If the limit (\ref{def:Lyexp}) exists we call it a \emph{Lyapunov
exponent} of $\mathcal A$.  We refer the reader to the expository article of Wilkinson \cite{Wi} for an introduction to Lyapunov exponents.

\medskip

When $X$ is a finite measure space, subject to various measurability and
integrability conditions, the
Oseledets Theorem \cite{O} says
that  for all $v\in\V$ the limit (\ref{def:Lyexp})
exists almost surely and coincides with  one of the real numbers 
$$
\lambda_1\geq \lambda_2\geq \cdots\geq \lambda_n.
$$
(See also Gol'dsheid-Margulis\cite{GM}, Guivarc'h-Raugi\cite{GR},
Ruelle\cite{Ru}, and Viana\cite{V}.)

Recall that we have set $\psi^+(x)=\max(0,\psi(x))$ for a real-valued function $\psi.$ 
Then the theorem of Pesin \cite{P} and Ruelle \cite{Ru2} implies that in the setting of Example \ref{ex:derivative}, if $f:M\to M$  preserves a
measure $\mu$, absolutely continuous with respect to Lebesgue, and
$\mathcal A$ is the derivative cocycle, we have 
\begin{equation}
  \label{eq:lyapunov_integral}
\int_M \sum_i \lambda_i^+(x)\, dx= h_\mu(f),
\end{equation}
where $h_\mu(f)$ is the entropy of $f$ with respect to $\mu.$
From a dynamical systems perspective,  knowing when $h_\mu(f)$ is
positive and how large it may be is of great interest. 
But the Lyapunov exponents of the derivative cocycle are generally
difficult to compute, even to show positivity of the integral
(\ref{eq:lyapunov_integral}).
On the other hand the Lyapunov exponents of a \emph{random product} are
frequently easy to show positive.

One attempt to approach the problem is to consider diffeomorphism or
more generally cocycles that belong to \emph{rich} families $\mathscr
F$, and to prove that
$\int_M\sum_i \lambda^+(x,f)\,dx$  is positive for at least some elements of the
family by comparing with Lyapunov exponents of random products. 
It is not clear what the notion of \emph{rich} should be to carry out
this program of bounding the average Lyapunov exponents by those of random products.  

There is some success in Pujals-Robert-Shub\cite{PRS},
Pujals-Shub\cite{PS}, De la Llave-Shub-Sim\'o\cite{DSS}, and
Dedieu-Shub\cite{DS}, and an extensive
discussion in Burns-Pugh-Shub-Wilkinson\cite{BPSW} for derivative
cocycles.
A notion of rich which comes close for the circle and two sphere is
$\On$ invariance.
The theorem of \cite{DS} for unitarily invariant measures on $\glnc$ was important in this direction.   

\medskip
\subsection{Outline of paper}
We conclude this introduction with an outline of the remainder of the paper and a sketch of the ideas used in the proof of Theorem \ref{thm:main}.  The sums $\sum_{i\leq k} r_i$ of the random Lyapunov exponents appearing in Theorem \ref{thm:main} admit a geometric interpretation relating them to an integral over the Grassmannian manifold $\grass nk$ of $k$ dimensional subspaces of $\R^n.$  We use this relation in Section \ref{thm:mainOn} to reduce the proof of Theorem \ref{thm:main} to a comparison of an integral on the the orthogonal group to an integral on the Grassmannian.  This comparison is effected by applying the coarea formula to the two projections $\Pi_1,\Pi_2$ of the manifold $\Vb_A$ of fixed $k$-dimensional subspaces
$$\Vb_A=\{(U,g)\in \On\times \grass nk:\, (UA)_{\#}g=g \} \mbox{ for fixed }
A\in\gln. 
$$
This use of the coarea formula, presented in Sections \ref{sec:VA} and \ref{sec:proofmain} is similar to the approach of \cite{DS}.  Our main point of departure from the earlier paper comes in Section \ref{sec:ointegrals} in our treatment of bounding an integral of the normal Jacobian of the projection $\Pi_1.$  We use the theory of spherical polynomials for the symmetric space $G/K$ for $G=\gln$ and $K=\OO_N(\R).$  Our Theorem \ref{thm:GS} is a consequence of a positivity results for Jack polynomials due to Knop-Sahi \cite{knop-sahi}.  
This approach highlights a difficulty in extending the results of \cite{DS} to our setting.  In the case of $G=\glnc, K=\UU_n(\C),$ the associated spherical polynomials are simply Schur polynomials, thus permitting a more direct treatment in the earlier work using the Vandermonde determinant, see Section 4.5 of \cite{DS}.

\medskip

We hope that the results and techniques of this paper stimulate further interactions between the ergodic theory of cocycles and harmonic analyses on symmetric spaces.  One appealing direction is the investigations of families
of cocycles which have elements with $\int_{x\in X}\sum_i\lambda_i^+(x)\,dx$
positive. 
Especially interesting would be more rich families of dynamical systems which must have some elements of positive entropy. One approach for measure preserving families of dynamical systems would be to compare the Lyapunov exponents  of the derivative cocycles of the family to the Lyapunov exponents of the random products of the cocycles of the family.

\section{Proof of Theorem \ref{thm:main}} 
\label{sec:proofth2} 

Let $\grass{n}{k}$ be the Grassmannian of $k$-dimensional subspaces of
$\R^n$.  Given $g\in\grass nk$, let $\Ort{g}$ be the subgroup of $\On$
that fixes $g$.  For $A\in\gln$ we denote by $A_\#$ the mapping corresponding to the natural induced action on $\grass nk$ and by $A|_g$ the restriction of $A$ to the subspace $g$.
\medskip

Consider the Riemannian metric on $\On$ coming from its embedding in the space of $n\times n$ matrices with the natural inner product $\langle A,B\rangle=\tr(A\trnsp B).$
As a Lie group, this Riemannian structure on $\On$ is
left and right invariant and it induces a Riemannian
structure on $\grass nk$ as a homogeneous space of $\On$. We denote by 
$\vol \OO_n(\R)$ and $\vol\, \grass nk$ the Riemannian volumes of the orthogonal group and Grassmannian respectively, and note the relation
\begin{equation}
  \label{eq:volGr}
\vol\,\grass nk=\frac{\vol\OO_n(\R)}{\vol\OO_{k}(\R)\cdot\vol\OO_{n-k}(\R)}.
\end{equation}
Define the constant
\begin{equation}
  \label{def:cnk}
c_{n,k}=\frac{\vol\OO_{k}(\R)\cdot\vol\OO_{n-k}(\R)}{\binom nk}.
\end{equation}

\begin{thm}\label{thm:mainOn}
For any $A\in\gln$   
we have
$$
    \int_{U\in\On} \left[\sup_{\substack{g\in\grass nk:\\ (UA)_\#g=g}}\left(\log^+|\det
  UA|_g|\right)\right]d\On 
  \geq 
c_{n,k}\int_{g\in\grass nk}\log^+|\det A|_g|\,d\grass nk.
$$

\end{thm}

If we integrate instead with respect to the Haar measure $dU$ on $\OO_n(\R)$ and the invariant probability measure $dg$ on $\grass nk$ we get 
\begin{equation}
  \label{mainHaar}
 \int_{U\in\On} \left[\sup_{\substack{g\in\grass nk:\\ (UA)_\#g=g}}\left(\log^+|\det
  UA|_g|\right)\right]dU
  \geq 
\frac{1}{\binom nk}
\int_{g\in\grass nk}\log^+|\det A|_g|\,dg.
\end{equation}
This follows immediately from Theorem \ref{thm:mainOn} and (\ref{eq:volGr}). 
The proof of Theorem \ref{thm:mainOn} is given
in Sections \ref{sec:proofth2} and \ref{sec:ointegrals}.


\medskip

Note that Theorem \ref{thm:mainOn} implies a slightly more general
result.

\begin{thm}\label{thm:gln}
  There is a constant $c_n>0$ such that, if $\mu$ is an orthgononally
  invariant probability measure on $\gln$, then,
    $$
    \int_{A\in\gln} \left[\sup_{\substack{g\in \grass nk:\\A_{\#}g=g}}\log^{+}|\det
  A|_g|\right]d\mu 
    \geq c_n\int_{A\in\gln} \int_{g\in \grass nk} \log^{+}|\det A|_g|\,d\grass
  nk\,d\mu.
$$
\end{thm}
\begin{proof}
  Disintegrate the measure along orbits and apply the inequality of
  Theorem \ref{thm:mainOn}, orbit
  by orbit. Note that $|\det A|_g|$ is constant on an orbit.
\end{proof}

\begin{proof}[Proof of Theorem \ref{thm:main}]
  
Pointwise we have 
  $$
  \left(\sum_{i=1}^k\log|\lambda_i(A)|\right)^+ \geq \sup_{\substack{g\in \grass
  nk:\\ A_{\#}g=g}}\log^{+}|\det A|_g|,\quad (A\in\gln),
  $$
where the supremum on the right hand side is defined to be $0$ if the
set of $g\in\grass nk$ such that $A_\#g=g$ is empty.

Then, for finishing the proof of  Theorem \ref{thm:main} it suffice to identify the right hand side of the
expression in Theorem \ref{thm:gln} in terms of
$\left(\sum_{i=1}^kr_i\right)^+$.

As in the proof of Theorem 3 in \cite{DS}, and surely in many other
places in the reference,
$$
 \sum_{i=1}^k r_i =\int_{A\in\gln}\int_{g\in \grass nk}\log \Big|\det
  A|_g\Big|\, d\grass nk\,d\mu
$$
so
\begin{align*}
  \left(\sum_{i=1}^k r_i\right)^+ &= \left(\int_{A\in\gln}\int_{g\in \grass nk}\log |\det
  A|_g|\, d\grass nk\,d\mu\right)^+\\
  &  \leq \int_{A\in\gln}\int_{g\in \grass nk}\log^+ |\det
  A|_g|\, d\grass nk\,d\mu.
\end{align*}
\end{proof}

We will give the proof of Theorem \ref{thm:mainOn} in Sections \ref{sec:proofth2} and \ref{sec:ointegrals} after some preparations in
the next section.

\section{Manifold of fixed subspaces}
\label{sec:VA}

Let $A\in \gln$, and define the manifold of fixed $k$-dimensional
subspaces 
$$
\Vb_A:=\{(U,g)\in \On\times \grass nk:\, (UA)_{\#}g=g \}
$$

Let $\Pi_1:\Vb_A\to \On$ and $\Pi_2:\Vb_A\to\grass nk$ be the associated
projections.

$$
\xymatrix{
&\Vb_A\ar@{->}@/_/[ddl]_{\Pi_1}\ar@{->}@/^/[ddr]^{\Pi_2}&\\
&&\\
\On&&\grass nk}
$$

\medskip

Given $g\in \grass nk$, one has $$
\Pi_2^{-1}(g)=\{(U,g): U\in \On,\, A_{\#}g=(U^{-1})_{\#}g\}.
$$
By abusing notation we will write 
$\Pi_2^{-1}(g)=\Pi_1\Pi_2^{-1}(g)$, which can be identified with the product space  $\Ort{g}\times\Ort{g^\perp}$,
which we in turn identify with $\Ort{k}\times \Ort{n-k}$.
Similarly, given $U\in \On$, we identify  
$
\Pi_1^{-1}(U)$ with $$\{g\in \grass nk:\mbox{ fixed by } (UA)_\#\}.
$$

  \begin{rem}\label{rem:finitefiber}
  Note that, on a set of full measure in $\On$, the fiber $\Pi_1^{-1}(U)$ is finite
  and $\#\Pi_1^{-1}(U)$ is bounded above by $\binom nk$.
  This follows from the fact that the set of $U\in\On$ such that
  $UA$ has repeated eigenvalues, is a proper subvariety of
  $\On$ defined by the discriminant of the characteristic polynomial of
    $UA$. Therefore a $k$-dimensional invariant subspace for $UA$, where 
  $U$ lies in the complement of the algebraic subvariety described above, corresponds to a choice of $k$-eigenvalues for $UA$, and corresponding
eigenspaces,
\end{rem}


The tangent space to the Grassmannian $\grass nk$ at $g$,  can be
identified in a natural way with  the set of  linear maps
$\Hom(g,g^\perp)$, i.e.,  any subspace $g'\in\grass nk$, in a neighborhood of $g$
can be represented as the graph of a unique map in $\Hom(g,g^\perp)$.
More precisely, if we denote by $\pi_g$ and $\pi_{g^\perp}$ the
orthogonal projections of $\R^n=g\oplus g^\perp$ into $g$ and $g^\perp$
respectively, then, $g'\in \grass nk$
such that $g'\cap g^\perp=\{0\}$, is the graph of the linear map
$\pi_{g^\perp}\circ((\pi_{g})|_{g'})^{-1}$.

\begin{lem}\label{lem:endgrass}
  Let $B\in\gln$, and $g\in\grass nk$ such that $B_\#g=g$. Then, the
  induced map  $\mathcal{L}_B:\Hom(g,g^\perp)\to \Hom(g,g^\perp)$, on
  local charts, is given by
  $$
  \mathcal{L}_B(\varphi)= [\pi_{g^\perp}(B|_{g^\perp})]\circ \varphi\circ
  \left([\pi_{g}(B|_{g})]+ [\pi_{g}(B|_{g^\perp})]\circ \varphi)
  \right)^{-1}
  $$
Furthermore,
its derivative at $g$, represented by $0\in \Hom(g,g^\perp)$, is given
  by
  $$
  D\mathcal{L}_B(0)\dot\varphi=
  [\pi_{g^\perp}(B|_{g^\perp})]\circ\dot\varphi\circ
  [\pi_{g}(B|_{g})]^{-1}
  $$
\end{lem}

Let us denote by $\NJ_{\Pi_1}$ and $\NJ_{\Pi_2}$ the normal Jacobians of
the maps $\Pi_1$ and $\Pi_2$ respectively. (See \cite[Section 3.1]{DS}.)

\begin{lem}[{\cite[Section 3.2]{DS}}]\label{lemma:ds}
  Given $(U,g)\in\mathbb V_A$, one has 
 \begin{itemize}
   \item ${\NJ}_{\Pi_1}(U,g)= |\det \Id -D\mathcal{L}_{UA}(g)|$;
   \item $\NJ_{\Pi_2}(U,g)=1$.
 \end{itemize} 
\end{lem}

In Section 5 we will need the normal Jacobian written more explicitly.  To this end, choose bases $v_1, \ldots, v_k$ for $g$ and $v_{k+1}, \ldots, v_n$ for its orthogonal complement $g^\perp$.  In terms of the basis $v_1,\ldots, v_n$ of $\R^n$, a linear map $B:\R^n\to\R^n$ which satisfies $Bg=g$ is represented by a matrix of the form 
$$
\left(\begin{array}{ c | c }
    B_1 & * \\
    \hline
    0 & B_2
  \end{array}\right).
$$
By Lemma \ref{lem:endgrass}, if $X$ is the matrix representing  $\dot\varphi$ in this basis, then 
$D\mathcal{L}_B(0)\dot\varphi$ is represented by the matrix $B_2XB_1^{-1}$.

\begin{lem}\label{lem:jacobian}
Let  $(U,g)\in\mathbb V_A$ and let 
$$
\left(\begin{array}{ c | c }
    B_1 & * \\
    \hline
    0 & B_2
  \end{array}\right).
$$
represent the map $UA$ in the basis $v_1,\ldots, v_n$ defined above.  Then 
$$\det (\Id -D\mathcal{L}_{UA}(g))=
\det(\Id-B_2\otimes \trnsp B_1^{-1})
$$
\end{lem}

\section{Proof of Theorem \ref{thm:mainOn}}\label{sec:proofmain}

Let $\phi:\grass nk\to\R$ be an integrable  function, and
let $\hat\phi:\mathbb{V}_A\to\R$ be its lift to $\mathbb V_A$, i.e.
$\hat\phi$ is  given by
$\hat\phi:=\phi\circ\Pi_2$. (Note that given $g\in\grass nk$,  $\hat\phi$ is constant in the
fiber $\Pi_2^{-1}(g)$, and its value coincides with the value of
$\phi$ at $g$.) 

For a set of full measure of $U\in \On$ (cf. Remark
\ref{rem:finitefiber}) we have 
\begin{equation}\label{eq:supleqav}
  \binom{n}{k}\sup_{g\in\Pi_1^{-1}(U)}(\phi(g))\geq
  \#\Pi_1^{-1}(U)\sup_{g\in\Pi_1^{-1}(U)}(\phi(g))\geq
\sum_{g\in\Pi_1^{-1}(U)}\phi(g).
\end{equation}

By the coarea formula applied to $\Pi_1$ we get
\begin{equation}\label{eq:coarea1}
  \int_{U\in\On}\left( \sum_{g\in\Pi_1^{-1}(U)}\phi(g)\right)d\On=
  \int_{\Vb_A}\hat\phi(U,g)\, \NJ_{\Pi_1}(U,g)\,
  d\Vb_A
\end{equation}
On the other hand, applying the coarea formula to the projection $\Pi_2$, 
\begin{align}\label{eq:coarea2}
  &\int_{\Vb_A}\hat\phi(U,g)\, \NJ_{\Pi_1}(U,g)\, d\Vb_A 
  \\& \qquad\qquad
  =\int_{g\in\grass nk}\left(\int_{U\in \Pi_2^{-1}(g)}
  \phi(g)\, \NJ_{\Pi_1}(U,g)\, d\Pi_2^{-1}(g) \right)\,d\grass
  nk,\nonumber
\end{align}
where we have used the fact that $\NJ_{\Pi_2}=1$.

Then from (\ref{eq:supleqav}), (\ref{eq:coarea1}), (\ref{eq:coarea2}) and Lemma \ref{lemma:ds}
we have 
\begin{align}\label{eq:final}
  &  \int_{U\in\On}
  \left(\sup_{g\in\Pi_1^{-1}(U)}(\phi(g))\right)d\On\nonumber\\ \nn
  & \qquad\qquad \geq
  \binom nk^{-1} \int_{g\in\grass nk}\phi(g)\left[\int_{U\in \Pi_2^{-1}(g)}
  {\NJ_{\Pi_1}(U,g)}\, d\Pi_2^{-1}(g) \right]\,d\grass nk\\
& \qquad\qquad =\binom nk^{-1} \int_{g\in\grass nk}\phi(g)\left[\int_{U\in \Pi_2^{-1}(g)}|\det \Id -D\mathcal{L}_{UA}(g)|
  \, d\Pi_2^{-1}(g) \right]\,d\grass nk
\end{align}

Specialize now to  $\phi:\grass nk\to\R$ given by
\begin{equation*}
  \phi(g):=\log^+ |\det A|_g|,\quad g\in \grass nk.
\end{equation*}
In particular 
$$
\sup_{g\in\Pi_1^{-1}(U)}\phi(g) =  \sup_{\substack{g\in\grass nk:\\(UA)_\#g=g}}\log^+ |\det (UA)|_g|.
$$

Now, the proof of Theorem \ref{thm:mainOn}, follows from
Theorem \ref{thm:GS} below which is used to bound the bracketed inner integral in 
\eqref{eq:final}; this together with the nonnegativity of $\phi$ proves
Theorem 2.

\begin{thm}\label{thm:GS}
Given $g\in\grass nk$, one has 
  $$
\int_{U\in \Pi_2^{-1}(g)}
 \left( \det \Id -D\mathcal{L}_{UA}(g)\right)\, d\Pi_2^{-1}(g)(U) \geq \vol\OO_k(\R)\cdot\vol\OO_{n-k}(\R).
  $$
\end{thm}
The proof is given in the following section.

\section{Proof of Theorem \ref{thm:GS}}
\label{sec:ointegrals}

For fixed $g\in\grass nk$ choose $U_0\in \On$ such that $U_0Ag=g$.  Then 
$$\Pi_2^{-1}(g)=\{VU_0:V\in\Ok\times\Onk \},$$
where we continue to identify $\Ok\times\Onk$ with $\Ort g\times\Ort {g^\perp} $.  We have
\begin{align*}
 \int&_{U\in \Pi_2^{-1}(g)}
  \det\left( \Id -D\mathcal{L}_{UA}(g)\right)\, d\Pi_2^{-1}(g)(U) \\
&=
\int_{V\in\Ok\times\Onk}\det(\Id-D\mathcal{L}_{VU_0A}(g))\, d\Pi_2^{-1}(g)(VU_0)\\
 &=\vol\OO_k(\R)\cdot\vol\OO_{n-k}(\R)\\
& \qquad\qquad\times\int_{\psi_1\in\Ok}\int_{\psi_2\in\Onk}
\det(\Id-(\psi_2B_2)\otimes\trnsp(\psi_1B_1)^{-1})\, d\psi_2\,d\psi_1\\
\end{align*}
where $d\psi_1,d\psi_2$ are the Haar measures on $\Ok$ and $\Onk$
The last equality follows from Lemma \ref{lem:jacobian}, with 
$$ B_1=\pi_g((U_0A)|_g)\mbox{ and }
B_2=\pi_{g^\perp}((U_0A)|_{g^\perp})
$$
More generally, for $B_1\in\GL_k(\R), B_2\in\GL_{n-k}(\R)$ we consider the integral of the characteristic polynomial
\begin{equation}
  \label{eq:defI}
\I(B_1,B_2;u)=\int_{\psi_1\in \Ok}\int_{\psi_2\in \Onk}
\det\left(\Id-u(\psi_2B_2)\otimes\trnsp{(\psi_1 B_1)^{-1}}\right)\,d\psi_2\,d\psi_1.
\end{equation}
Therefore Theorem \ref{thm:GS} is equivalent to 
\begin{equation}
  \label{eq:I}
\I(B_1,B_2;1)\geq 1.
\end{equation}
In fact we will prove an explicit formula for the integral, expressing the coefficients of the characteristic polynomial $\I(B_1,B_2;u)$ as polynomials in the squares of the singular values of $B_1$ and $B_2^{-1}$ with positive integer coefficients.

We complete the proof of the Theorem \ref{thm:GS} and inequality \eqref{eq:I} in several steps.  First we use the representation theory of the general linear group to factor the double integral into a linear combination of a product of two integrals over $\On$ and $\Onk$ respectively.  Next each orthogonal group integral is identified with a \emph{spherical polynomial}.  Finally, the theorem follows from an identity between spherical polynomials and Jack polynomials, and a positivity result for the latter due to Knop and Sahi \cite{knop-sahi}.

\subsection{Orthogonal group integrals}

We begin by expanding the characteristic polynomial in the integrand
as a sum of traces:
\begin{equation*}
\det\left(\Id-u(\psi_2 B_2)\otimes\trnsp{(\psi_1 B_1)^{-1}}\right)
=\sum_{j=0}^{k(n-k)} (-u)^j \tr\bigwedge\nolimits^j(\psi_2 B_2\otimes \trnsp(\psi_1 B_1)^{-1}) \\
\end{equation*}
Next,  decompose the exterior powers of the tensor product as
\begin{equation}
  \label{eq:2}
\bigwedge\nolimits^j(\psi_2 B_2\otimes \trnsp(\psi_1 B_1)^{-1}) =
\sum_{\lambda:|\lambda|=j} \rho_{\lambda'}(\psi_2 B_2)\otimes \rho_{\lambda}(\trnsp(\psi_1 B_1)^{-1}),
\end{equation}
where 
\begin{itemize}
\item the sum is over all partitions $\lambda$ of $j$ with at most $k$
rows and $n-k$ columns,
\item  $\lambda'$ is the partition conjugate to
$\lambda$ and
\item $\rho_\lambda, \rho_{\lambda'}$ are the irreducible representations of $\GL_k(\R)$ and $\GL_{n-k}(\R)$ associated to the partitions $\lambda,\lambda',$ respectively.
\end{itemize}
See, for example, Exercise 6.11 of Fulton-Harris \cite{FH}.  
Since the trace of a tensor product of two matrices is the
product of the two traces, we may write
\begin{equation*}
\det\left(\Id-u(\psi_2 B_2)\otimes\trnsp{(\psi_1 B_1)^{-1}}\right)
=\sum_{j=0}^{k(n-k)} (-u)^j\sum_{|\lambda|=j} 
\tr\rho_{\lambda'}(\psi_2 B_2)\cdot\tr\rho_\lambda (\psi_1 B_1) .
\end{equation*}
Integrating over $\OO_k(\R)\times\OO_{n-k}(\R)$ we find that $\I(A,B;u)$ is equal to 
\begin{align} \label{eq:3}
\sum_{j=0}^{k(n-k)}(-u)^j \sum_{|\lambda|=j}&
\left(\int_{\psi_2\in \OO_{n-k}(\R)}\tr\rho_{\lambda'}(\psi_2 B_2)\,d\psi_2\right)\cdot  
\left(\int_{\psi_1\in \OO_{k}(\R)} \tr\rho_{\lambda}(\psi_1 B_1)\,d\psi_1\right). 
\\ \nn
&= 1+ \sum_{\substack{1\leq j\leq k(n-k)\\ |\lambda| = j}}
(-u)^j F_{\lambda'}(B_2)F_{\lambda}(B_1) ,
\end{align}
where for $M\in \GL_N(\R)$ and $\mu$ a partition of $j$ with at most $N$ parts, we define
\begin{equation}
  \label{def:F}
F_\mu(M)=\int_{\psi\in \OO_N(\R)}\tr\rho_\mu(\psi M)\,d\psi
\end{equation}
Theorem \ref{thm:GS} follows from the following more explicit result.

\begin{thm}\label{thm:Fpositivity}
  Let $M\in \GL_N(\R)$ and $\mu=(\mu_1, \ldots, \mu_r)$ with $\mu_1\geq\mu_2\geq\cdots\geq\mu_r>0$ be a partition of $k$ of at most $N$ parts.  
  \begin{enumerate}
  \item If any of the parts $\mu_i$ is odd, then $F_\mu(M)=0.$
  \item If all the parts $\mu_i$ are even, then $F_\mu(M)$ is an even polynomial in the singular values of $M$ with positive coefficients.
  \end{enumerate}
\end{thm}


\subsection{Spherical polynomials}

The proof of Theorem  \ref{thm:Fpositivity} involves the theory of spherical polynomials for the symmetric space $G/K$ where $G=\GL_N(\R)$ and $K=\OO_N(\R)$, and Jack polynomials. We recall these briefly.

Let $\cP_N$ be the set of partitions with at most $N$ parts, thus
\[ \cP_N=\{\mu \in \Z^N\mid \mu_1\ge \mu_2\ge \cdots\ \geq\mu_N\ge 0\}. \] 

For $\mu \in \cP_N$, let $(\rho_\mu, V_\mu)$ be the corresponding representation of $G$, and let $V_\mu^*$ be the contragredient representation.  A matrix coefficient of $V_\mu$ is a function on $G$ of the form
\[ \phi_{u,v}(M) =\langle u, \rho_\mu(M)v \rangle, \] 
where $u\in V_\mu^*$ and $v\in V_\mu$. We write $\F_\mu$ for the span of matrix coefficients of $V_\mu$. Then $\F_\mu$ is stable under left and right multiplication by $G$, and one has a $G\times G$-module isomorphism
\[  V_\mu^*\otimes  V_\mu \approx \F_\mu,\quad u\otimes v\mapsto \phi_{u,v}.\]

\begin{thm} \label{thm:spherical} Let $\mu$ be a partition in $\cP_N$. Then the following are equivalent
\begin{enumerate}
\item $\mu$ is even, that is, $\mu_i\in 2\Z$ for all $i$.
\item $V_\mu$ has a spherical vector, that is, a vector fixed by $K$.
\item $V_\mu^*$ has a spherical vector.
\item  $\F_\mu$ contains a spherical polynomial $\phi_\mu$, that is, a function satisfying
\[ \phi_\mu(kgk')=\phi_\mu(g), \; g\in G,\; k,k'\in K.\]
\end{enumerate}
The spherical vector $v_\mu$ and spherical polynomial $\phi_\mu$ are unique up to scalar multiple, and the latter is usually normalized by the requirement $\phi_\mu(e)=1$, which fixes it uniquely.
\end{thm} 
\begin{proof}
This follows from the Cartan-Helgason theory of spherical representations \cite[Theorem~V.4.1]{He}.
\end{proof}

We now connect the polynomial $F_\mu$ to $\phi_\mu$.
 
\begin{thm}\label{thm:Fphi}
Let $F_\mu(M)$ be as in \eqref{def:F}. If $\mu$ is even then $F_\mu =\phi_\mu$, otherwise $F_\mu=0$.
\end{thm}
\begin{proof}
If $\{v_i\},\{u_i\}$ are dual bases for $V_\lambda,V_\lambda^*$ then $\tr \rho_\mu(M)=\sum_i\phi_{u_i,v_i}(M)$, thus the character $\chi_\mu(M) = \tr \rho_\mu(M)$ is an element of $\F_\mu$. Since $\F_\mu$ is stable under the left action of $K$, it follows that $F_\mu(M)=\int_K\chi_\mu(k M)\,dk$ is in $\F_\mu$ as well. 

We next argue that $F_\mu$ is $K\times K$ invariant. For this we compute as follows:
 \[ F_\mu(k_1Mk_2)=\int_K\chi_\mu(kk_1Mk_2)\,dk = \int_K\chi_\mu(k_2kk_1M)\,dk =  F_\mu(M)\] 
Here the first equality holds by definition, the second is a consequence the invariance of the trace character -- $\chi_\mu(AB)=\chi_\mu(BA)$, and the final equality follows from the $K\times K$ invariance of the Haar measure $dk$.

By Theorem \ref{thm:spherical} this proves that $F_\mu$ is a multiple of $\phi_\mu$ if $\mu$ is even, and $F_\mu=0$ otherwise. To determine the precise multiple we need to compute the following integral for even $\mu$ :
\[ F_\mu(e)=\int_K\chi_\mu(k)\,dk.\] 
By Schur orthogonality, this integral is the multiplicity of the trivial representation in the restriction of $V_\mu$ to $K$, which is $1$ if $\mu$ is even. Thus we get $F_\mu=\phi_\mu$, as desired. 
\end{proof} 

\subsection{Jack polynomials}
Jack polynomials $J_\lambda^{(\alpha)}(x_1,\ldots,x_N)$  are a family of symmetric polynomials in $N$ variables whose coefficients depend on a parameter $\alpha$. The main result of \cite{knop-sahi} is that these coefficients are themselves positive integral polynomials in the parameter $\alpha$. 

Spherical functions correspond to Jack polynomials with $\alpha=2$. More precisely, we have
\begin{equation}\label{=J2}
 \phi_\mu(g)= \frac{J_\lambda^{(2)}(a_1,\ldots, a_N)}{J_\lambda^{(2)}(1,\ldots, 1)}, \quad \mu=2\lambda
 \end{equation} 
where $a_1,\ldots,a_N$ are the eigenvalues of the symmetric matrix $g^Tg$; in other words, the $a_i$ are the squares of the singular values of $g$,

We can now finish the proof of Theorem \ref{thm:Fpositivity}.

\begin{proof}[ Proof of Theorem \ref{thm:Fpositivity}]
Part (1) follows from Theorem \ref{thm:Fphi}. Part (2) follows from formula \eqref{=J2} and the positivity of Jack polynomials as proved in \cite{knop-sahi} 
\end{proof}

\subsection{Examples} %
\label{subsec:examples}
We conclude this section with two low rank examples of the characteristic polynomials $\I(A,B;u)$ for $A\in \GL_k(\R), B\in \GL_{n-k}(\R).$  As we may assume $A$ and $B$ are diagonal, let us write
$$A=\diag(a_1,\ldots, a_k)\mbox{ and } B=\diag(b_1, \ldots, b_k).
$$

\subsubsection*{The case $n=4,k=2$}
Here we consider the integral
\begin{equation}
  \label{eq:case22}
\I(A,B;u)=\int_{\psi_1\in \OO_2(\R)}\int_{\psi_2\in \OO_2(\R)}
\det\left(\Id-u(\psi_2B)\otimes\trnsp{(\psi_1 A)^{-1}}\right)\,d\psi_2\,d\psi_1.
\end{equation}
As we are essentially integrating over the circle, it is easy to compute this directly and see that
\begin{equation}
  \label{eq:case22ans}
\I(A,B;u)=1+\frac{\det(B)^2}{\det(A)^2}u^4=1+\frac{b_1^2b_2^2}{a_1^2a_2^2}u^4.
\end{equation}

\subsubsection*{The case $n=6,k=2$}
In this case we use (\ref{eq:3}) to compute
\begin{equation}
  \label{eq:case42}\nn
\I(A,B;u)=\int_{\psi_1\in \OO_2(\R)}\int_{\psi_2\in \OO_4(\R)}
\det\left(\Id-u(\psi_2B)\otimes\trnsp{(\psi_1 A)^{-1}}\right)\,d\psi_2\,d\psi_1
\end{equation}
for $A\in \GL_2(\R),B\in\GL_4(\R).$  Write
$$\I(A,B;u)=1+c_2u^2+c_4u^4+c_6u^6+c_8u^8.
$$
By part 1 of Theorem \ref{thm:Fpositivity} we immediately see that $c_2=c_6=0$ because there are no partitions $\lambda$ of $2$ or $6$ for which both $\lambda$ and its conjugate $\lambda'$ have only even parts.  The only even partition of $k=8$ with at most 2 parts and with even conjugate is $\lambda=(4,4).$  For $V$ the standard two-dimensional representation of $\GL_2(\R)$, we have that $\rho_\lambda(V)=\sym^4(\Lambda^2 V)$
is the fourth power of the determinant representation.  Hence
$$F_\lambda(A^{-1})=\det A^{-4}.$$
Similarly for $W$ the standard four-dimensional representation of $\GL_4(\R)$, the conjugate $\lambda'=(2,2,2,2)$ and $\rho_{\lambda'}(W)=\sym^2(\Lambda^4(W))$ is the square of the determinant.  Hence
$F_{\lambda'}(B)=\det B^{2}$
and
$$c_8=\frac{\det(B)^2}{\det(A)^4}.$$

The only even partition of $k=4$ with even conjugate is $\lambda=\lambda'=(2,2).$  In this case $\rho_\lambda(V)$ is the square of the determinant representation.  The dimension 20 representation $\rho_\lambda(W)$ is a quotient of $\sym^2(\Lambda^2(W))$ with a unique $\OO_4(\R)$-fixed vector, namely, the image of
$$v=(e_1\wedge e_2)^2+(e_1\wedge e_3)^2+(e_1\wedge e_4)^2+(e_2\wedge e_3)^2+(e_2\wedge e_4)^2+(e_3\wedge e_4)^2.
$$
It is readily seen that the trace $\rho_\lambda(B)$ restricted to the span of $v$ is
$$\sum_{1\leq i< j\leq 4} b_i^2b_j^2.
$$
Then, including the normalizing factor of $1/J^{(2)}_{(1,1)}(1,1,1,1)=1/6$ we conclude that 
$$c_4=\frac{\frac16 (b_1^2b_2^2+b_1^2b_3^2+b_1^2b_4^2+b_2^2b_3^2+b_2^2b_4^2+b_3^2b_4^2)}
{a_1^2a_2^2}.$$

\bibliographystyle{alpha}
\bibliography{random_mean_exp}

\end{document}